\documentclass[UTF8,fontset=windows,final, leqno]{siamltex}
\usepackage{ctex}
\usepackage{amsmath}
\allowdisplaybreaks[4]
\usepackage{epsfig}
\usepackage{graphicx}
\usepackage{amssymb}
\usepackage{CJK}
\usepackage{color}
\usepackage[noadjust]{cite}
\usepackage{caption}

\numberwithin{equation}{section}
\newtheorem{remark}{Remark}[section]

\renewcommand\abstractname{Abstract}

\renewcommand\figurename{Figure}
\renewcommand\tablename{Table}

\renewcommand\refname{References}
\def\lam{{\lambda}}

\def\Ome{{\Omega}}

\def\nab{{\nabla}}
\def\vepsi{{\varepsilon}}
\def\p{{\partial}}
\def\reff#1{\eqref{#1}}
\def\norm#1#2{\Vert\,#1\,\Vert_{#2}}

\def\vepsi{\varepsilon}

\def\no{{\nonumber}}

\def\div{{\mbox{\rm div\,}}}

\def\p{{\partial}}

\def\nab{\nabla}
\def\Ome{\Omega}
\def\lam{\lambda}

\newcommand{\bRM}{\mathbf{RM}}
\newcommand{\br}{\mathbf{r}}

\def\ba{\mathbf{a}}
\def\bb{\mathbf{b}}

\def\bbf{\mathbf{f}}
\def\bu{\mathbf{u}}

\def\bv{\mathbf{v}}
\def\bw{\mathbf{w}}

\def\bg{\mathbf{g}}
\def\bn{\mathbf{n}}
\def\bH{\mathbf{H}}
\def\bV{\mathbf{V}}
\def\bL{\mathbf{L}}

\def\bV{\mathbf{V}}

\def\bX{\mathbf{X}}

\def\R{\mathbb{R}}

\def\bx{{\bf x}}
\CTEXoptions[today=old]

\begin{document}
	
%	\pdfoutput=1

%    \ifpdf
%    \DeclareGraphicsExtensions{.pdf, .jpg, .tif, .png}
%    \else
%    \DeclareGraphicsExtensions{, .jpg}
%    \fi

%\title{Multiphysics Finite Element Methods for a Poroelasticity Model}
\title{Well-posedness of weak solution for a nonlinear poroelasticity model\footnote{Last update: \today}}

\author{
Zhihao Ge\thanks{School of Mathematics and Statistics, Henan University, Kaifeng 475004, P.R. China ({\tt zhihaoge@henu.edu.cn}).
	The work of this author was supported by the National Natural Science Foundation of China under grant No.11971150.}
\and
Wenlong He\thanks{School of Mathematics and Statistics, Henan University, Kaifeng 475004, P.R. China.}
}

%%%
\maketitle

%\vspace{-1.4in}
%\slugger{sinum}{200x}{xx}{x}{xxx--xxx}
%\vspace{1.1in}

\setcounter{page}{1}

%\begin{PII}
%S00000000000
%\end{PII}

%\large

\begin{abstract}
 In this paper, we study the existence and uniqueness of weak solution of a nonlinear poroelasticity model widely used in many fields such as geophysics, biomechanics, civil engineering, chemical engineering, materials science, and so on. To better describe the proccess of deformation and diffusion underlying in the original model, we firstly reformulate the nonlinear poroelasticity by a multiphysics approach£¬which transforms the nonlinear fluid-solid coupling problem to a fulid-fluid coupling problem. Then, we adopt the similar technique of proving the well-posedness of nonlinear Stokes equations to prove the existence and uniqueness of weak solution of a nonlinear poroelasticity model. And we strictly prove the growth, coercivity and monotonicity of the nonlinear stress-strain relation, give the energy estimates and use Schauder's fixed point theorem to show the existence and uniqueness of weak solution of the nonlinear poroelasticity model. Besides, we prove that the weak solution of nonlinear poroelasticity model converges to the nonlinear Biot's consolidation model as the constrained specific storage coefficient trends to zero. Finally, we draw a conclusion to summary the main results of this paper.
\end{abstract}

\begin{keywords}
Nonlinear poroelasticity; Multiphysics approach; Nonlinear Stokes equations; Schauder's fixed point theorem.
\end{keywords}
\begin{AMS}
35A01, %Stability and convergence of numerical methods
35B45, %Error bounds
86A25, %Finite elements, Rayleigh-Ritz and Galerkin methods, finite methods
\end{AMS}

\pagestyle{myheadings}
\thispagestyle{plain}
\markboth{ZHIHAO GE, WENLONG HE}{WELL-POSEDNESS OF WEAK SOLUTION FOR NONLINEAR POROELASTICITY}

%%%%%%%%%%%%%%%%%%%%%%%%%%%%%%

\section{Introduction}\label{sec-1}
In recent years, poroelasticity model is widely used in various fields such as geophysics, biomechanics, civil engineering, chemical engineering, materials science and so on, one can refer to \cite{2,3,6,7,8,10,biot,coussy04,de86}. Especially, in modern materials science, porous materials such as polymers and metal foams are of great significance in lightweight design and aircraft industry, one can refer to \cite{2,coussy04,hamley07} and so on. The poroelasticity model is classified into linear poroelasticity model and nonlinear poroelasticity model according to the linear or nonlinear constitutive relation (cf. \cite{20210820}). For linear poroelasticity, Schowalter provides the analysis of well-posedness of weak solution to a linear poroelasticity model in \cite{20210819}. Phillips and Wheeler propose and analyze a continuous-in-time linear poroelasticity model in \cite{pw07}. Besides, Feng, Ge and Li in \cite{fglarxiv,fgl14}, propose a multiphysics approach to reformulate the linear poroelasticity model to a fluid-fluid coupled system, which reveals the underlying deformation and diffusion processes of the original model. In this paper, following the idea of \cite{fgl14}, we deal with the nonlinear poroelasticity model with the constitutive relation $\tilde{\sigma}(\mathbf{u})=\mu\tilde{\varepsilon}(\mathbf{u})+\lambda tr(\tilde{\varepsilon}(\mathbf{u}))\mathbf{I}$, where the deformed Green strain tensor is $ \tilde{\varepsilon}(\bu)=\dfrac{1}{2}(\nabla\bu+\nabla^{T}\boldsymbol{u}+2\nabla^{T}\bu\nabla\bu)$. Using the Cauchy-Schwarz inequality, Korn's inequality and other inequalities (see Section \ref{sec-4}), we prove the growth, coercivity and monotonicity of $\mathcal{N}(\nabla\bu)$ (see \reff{eq210823-1}), then we give the energy estimates and use Schauder's fixed point theorem to show the existence and uniqueness of weak solution of the nonlinear poroelasticity model. Besides, we prove that the weak solution of the nonlinear poroelasticity model converges to a nonlinear Biot's consolidation model as the constrained specific storage coefficient trends to zero. To the best of our knowledge, it is the first time to prove the the existence and uniqueness of weak solution based on a multiphysics approach without any assumption on the nonlinear stress-strain relation. Moreover, we find out that the multiphysics approach is key to propose a stable numerical method for the nonlinear poroelasticity model, and we will present the main results about numerical method for the nonlinear poroelasticity model in the future work.
 
The remainder of this article is organized as follows. In Section \ref{sec-2}, we reformulate the original model based on a multiphysics approach to a fluid-fluid coupling system. In Section \ref{sec-3}, we give the definition of weak solution to the original model and the reformulated model. In Section \ref{sec-4}, we prove the growth, coercivity and monotonicity based on a multiphysics approach without any assumption on the nonlinear stress-strain relation, and we use the energy estimates and Schauder's fixed point theorem to prove  the well-posedness of weak solution of the nonlinear  poroelasticity model. Besides, we prove that the nonlinear poroelasticity model converges to the nonlinear Biot's consolidation model as the constrained specific storage coefficient trends to zero. Finally, we draw a conclusion to summary the main results of this paper.

\section{PDE model and  multiphysics approach}\label{sec-2}
In this paper, we consider the following quasi-static poroelasticity model (for the linear case, one can refer to \cite{pw07,fgl14,fglarxiv}):
\begin{alignat}{2}\label{2.6} 
	-\div\tilde{\sigma}(\bu) + \alpha \nab p &= \bbf
	&&\qquad \mbox{in } \Ome_T:=\Ome\times (0,T)\subset \mathbf{\R}^d\times (0,T),\\
	(c_0p+\alpha \div \bu)_t + \div \bv_f &=\phi &&\qquad \mbox{in } \Ome_T,\label{2.7}
\end{alignat}
where
\begin{align} 
	&\tilde{\sigma}(\mathbf{u})=\mu\tilde{\varepsilon}(\mathbf{u})+\lambda tr(\tilde{\varepsilon}(\mathbf{u}))\mathbf{I},~~~~~ \tilde{\varepsilon}(\bu)=\dfrac{1}{2}(\nabla\bu+\nabla^{T}\bu+2\nabla^{T}\bu\nabla\bu), \label{2021}\\
	&\bv_f:= -\frac{K}{\mu_f} \bigl(\nab p -\rho_f \bg \bigr). \label{2.8}
\end{align}
Here $\bu$ denotes the displacement vector of the solid and $p$ denotes the pressure of the solvent. $\mathbf{I}$ denotes the $d\times d$ identity matrix and $\tilde{\vepsi}(\bu)$ is
known as the deformed Green strain tensor. $\bbf$ is the body force. The permeability tensor $K=K(x)$ is assumed to be symmetric and uniformly positive definite in the sense that there exists positive constants $K_1$ and $K_2$ such that
$K_1|\zeta|^2\leq K(x)\zeta\cdot \zeta \leq K_2 |\zeta|^2$ for a.e. $x\in\Omega$ and $\zeta\in \mathbf{\R}^d$; the solvent viscosity $\mu_f$, Biot-Willis constant
$\alpha$, and the constrained specific storage coefficient
$c_0$. In addition, $\tilde{\sigma}(\bu)$ is called the (effective) stress tensor. $\bv_f$ is the volumetric solvent flux and (\ref{2.8}) is called the well-known Darcy's law. $ \lambda $ and $ \mu $ are
Lam\'e constants, $\widehat{\sigma}(\bu, p):=\tilde{\sigma}(\bu)-\alpha p \mathbf{I}$ is the total stress tensor. We assume that $\rho_f\not\equiv 0$, which is a realistic assumption.

To close the above system, we set the following boundary and initial conditions in this paper:
 \begin{alignat}{2} \label{2.16}
 	\widehat{\sigma}(\bu,p)\bn=\tilde{\sigma}(\bu)\bn-\alpha p \bn &= \bbf_1
 	&&\qquad \mbox{on } \p\Ome_T:=\p\Ome\times (0,T),\\
 	\bv_f\cdot\bn= -\frac{K}{\mu_f} \bigl(\nab p -\rho_f \bg \bigr)\cdot \bn
 	&=\phi_1 &&\qquad \mbox{on } \p\Ome_T, \label{2.17} \\
 	\bu=\bu_0,\qquad p&=p_0 &&\qquad \mbox{in } \Ome\times\{t=0\}. \label{2.18}
 \end{alignat}
Introduce new variables
\[
q:=\div \bu,\quad \eta:=c_0p+\alpha q,\quad \xi:=\alpha p -\lam q.
\]
Denote 
\begin{eqnarray}
	\mathcal{N}(\nabla\bu)=\tilde{\sigma}(\bu)-\lambda \div\bu~ \mathbf{I},\label{eq210823-1}
\end{eqnarray}  
then we have
\begin{eqnarray}
	\mathcal{N}(\nabla\mathbf{u})=\mu\varepsilon(\mathbf{u})+\mu \nabla^{T}\mathbf{u}\nabla\mathbf{u}+\lambda\|\nabla\mathbf{u}\|_{F}^{2}\mathbf{I}.\label{eq210823-2}
\end{eqnarray}
Due to the fact of $(\nabla^{T}\mathbf{u}\nabla\mathbf{u},rot \mathbf{v})=0,~(\|\nabla\mathbf{u}\|_{F}^{2}\mathbf{I},rot \mathbf{v})=0$, so we have
 \begin{equation*}
 	(\mathcal{N}(\nabla\mathbf{u}),\nabla\mathbf{v})=(\mathcal{N}(\nabla\mathbf{u}),\varepsilon(\mathbf{v})),
 \end{equation*}
where $\varepsilon(\bu)=\dfrac{1}{2}(\nabla^{T}\bu+\nabla\bu)$.
 
In some engineering literature, Lam\'e constant
 $\mu$ is also called the {\em shear modulus} and denoted by $G$, and
 $B:=\lam +\frac23 G$ is called the {\em bulk modulus}. $\lam,~\mu$ and $B$
 are computed from the {\em Young's modulus} $E$ and the {\em Poisson ratio}
 $\nu$ by the following formulas
 \[
 \lam=\frac{E\nu}{(1+\nu)(1-2\nu)},\qquad \mu=G=\frac{E}{2(1+\nu)}, \qquad
 B=\frac{E}{3(1-2\nu)}.
 \]

 It is easy to check that
 \begin{align}\label{2.19}
 	p=\kappa_1 \xi + \kappa_2 \eta, \qquad q=\kappa_1 \eta-\kappa_3 \xi,
 \end{align}
 where $\kappa_1= \frac{\alpha}{\alpha^2+\lam c_0},
 	 \kappa_2=\frac{\lam}{\alpha^2+\lam c_0},
 	\kappa_3=\frac{c_0}{\alpha^2+\lam c_0}$.
 	
 Then the problem \reff{2.6}-\reff{2.8} can be rewritten as
 \begin{alignat}{2} \label{2.21}
 	-\div\mathcal{N}(\nabla\bu) + \nab \xi &= \bbf &&\qquad \mbox{in } \Ome_T,\\
 	\kappa_3\xi +\div \bu &=\kappa_1\eta &&\qquad \mbox{in } \Ome_T, \label{2.22}\\
 	\eta_t - \frac{1}{\mu_f} \div[K (\nab (\kappa_1 \xi + \kappa_2 \eta)-\rho_f\bg)]&=\phi
 	&&\qquad \mbox{in } \Ome_T. \label{2.23}
 \end{alignat}
 The boundary and initial conditions \reff{2.16}-\reff{2.18} can be rewritten as
 \begin{alignat}{2} \label{2.24}
 	\tilde{\sigma}(\bu)\bn-\alpha (\kappa_1 \xi + \kappa_2 \eta) \bn &= \bbf_1
 	&&\qquad \mbox{on } \p\Ome_T:=\p\Ome\times (0,T),\\
 	-\frac{K}{\mu_f} \bigl(\nab (\kappa_1 \xi + \kappa_2 \eta) -\rho_f \bg \bigr)\cdot \bn
 	&=\phi_1 &&\qquad \mbox{on } \p\Ome_T, \label{2.25} \\
 	\bu=\bu_0,\qquad p&=p_0 &&\qquad \mbox{in } \Ome\times\{t=0\}. \label{2.26}
 \end{alignat}
\begin{remark}\label{rem210823-1}
	It is now clear that $(\bu, \xi)$ satisfies a generalized nonlinear Stokes problem for
	a given $\eta$, and $\eta$ satisfies a diffusion problem for a given $\xi$. Thus, This new formulation reveals the underlying deformation and diffusion multiphysics process which occurs in the poroelastic material.
\end{remark}
  
\section{Definition of weak solution}\label{sec-3}
In this paper, $\Omega \subset \R^d \,(d=1,2,3)$ denotes a bounded polygonal domain with the boundary
$\p\Ome$. The standard function space notation is adopted in this paper, their precise definitions can be found in \cite{bs08,cia,temam}. In particular, $(\cdot,\cdot)$ and $\langle \cdot,\cdot\rangle$ denote respectively the standard $L^2(\Ome)$ and $L^2(\p\Ome)$ inner products. For any Banach space $B$, we let $\mathbf{B}=[B]^d$, and use $\mathbf{B}^\prime$ to denote its dual space. In particular, we use $(\cdot,\cdot)_{\small\rm dual}$ %and $\langle \cdot,\cdot \rangle_{\small\rm dual}$
to denote the dual product on $\bH^1(\Ome)' \times \bH^1(\Ome)$, and $\norm{\cdot}{L^p(B)}$ is a shorthand notation for
$\norm{\cdot}{L^p((0,T);B)}$.\\
We also introduce the function spaces
\begin{align*}
&L^2_0(\Omega):=\{q\in L^2(\Omega);\, (q,1)=0\}, \qquad \bX:= \bH^1(\Ome).
\end{align*}
From \cite{temam}, it is well known  that the following  inf-sup condition holds in the space $\bX\times L^2_0(\Ome)$:
\begin{align}\label{e2.0}
\sup_{\bv\in \bX}\frac{(\div \bv,\varphi)}{\norm{\bv}{H^1(\Ome)}}
\geq \alpha_0 \norm{\varphi}{L^2(\Ome)} \qquad \forall
\varphi\in L^2_0(\Ome),\quad \alpha_0>0.
\end{align}
Let
\[
\bRM:=\{\br:=\ba+\bb \times x;\, \ba, \bb, x\in \R^d\}
\]
denote the space of infinitesimal rigid motions. It is well known \cite{bs08, gra, temam} that $\bRM$ is the kernel of
the strain operator $\vepsi$, that is, $\br\in \bRM$ if and only if
$\vepsi(\br)=0$. Hence, we have
\begin{align}
\vepsi(\br)=0,\quad \div \br=0 \qquad\forall \br\in \bRM. \label{e4.100}
\end{align}
Let $\bL^2_\bot(\p\Ome)$ and $\bH^1_\bot(\Ome)$ denote respectively the subspaces of $\bL^2(\p\Ome)$ and $\bH^1(\Ome)$ which are orthogonal to $\bRM$, that is,
\begin{align*}
&\bH^1_\bot(\Ome):=\{\bv\in \bH^1(\Ome);\, (\bv,\br)=0\,\,\forall \br\in \bRM\},
\\
&\bL^2_\bot(\p\Ome):=\{\bg\in \bL^2(\p\Ome);\,\langle \bg,\br\rangle=0\,\,
\forall \br\in \bRM \}.
\end{align*}
It is well known \cite{dautray} that there exists a constant $c_1>0$ such that
\begin{eqnarray}
\inf_{\br\in \bRM}\|\bv+\br\|_{L^2(\Ome)}
\le c_1\|\vepsi(\bv)\|_{L^2(\Ome)} \qquad\forall \bv\in\bH^1(\Ome).\label{eq210823-11}
\end{eqnarray}
From \cite{fgl14}, we know that for each $\bv\in \bH^1_\bot(\Ome)$ there holds the following alternative version of the inf-sup condition
\begin{eqnarray}
\sup_{\bv\in \bH^1_\bot(\Ome)}\frac{(\div \bv,\varphi)}{\norm{\bv}{H^1(\Ome)}}
\geq \alpha_1 \norm{\varphi}{L^2(\Ome)} \qquad \forall
\varphi\in L^2_0(\Ome),\quad \alpha_1>0.\label{eq210823-6}
\end{eqnarray}

For convenience, we assume that $\bbf,~ \bbf_1,~ \phi$
and $\phi_1$ all are independent of $t$ in the remaining of the paper. We note that all the results of this paper can be easily extended to the case of time-dependent
source functions.
\begin{definition}\label{weak1}
Let $\bu_0\in\bH^1(\Ome),~ \bbf\in\bL^2(\Omega),~
\bbf_1\in \bL^2(\p\Ome),~ p_0\in L^2(\Ome),~ \phi\in L^2(\Ome)$,
and $\phi_1\in  L^2(\p\Ome)$.  Assume $c_0>0$ and
$(\bbf,\bv)+\langle \bbf_1,~ \bv \rangle =0$ for any $\bv\in \mathbf{RM}$.
Given $T > 0$, a tuple $(\bu,p)$ with
\begin{alignat*}{2}
&\bu\in L^\infty\bigl(0,T; \bH_\perp^1(\Ome)),
&&\qquad p\in L^\infty(0,T; L^2(\Omega))\cap L^2 \bigl(0,T; H^1(\Omega)\bigr), \\
&p_t, (\div\bu)_t \in L^2(0,T;H^{1}(\Ome)')
&&\qquad %c_0^{\frac12} p\in L^\infty\bigl(0,T; L^2(\Ome)),
\end{alignat*}
is called a weak solution to the problem \reff{2.6}--\reff{2.18}, if there hold for almost every $t \in [0,T]$
\begin{alignat}{2}\label{2.32}
&\bigl( \mathcal{N}(\nabla\bu), \vepsi(\bv) \bigr)
+\lam\bigl(\div\bu, \div\bv \bigr)
-\alpha \bigl( p, \div \bv \bigr)  && \\
&\hskip 2in
=(\bbf, \bv)+\langle \bbf_1,\bv\rangle
&&\quad\forall \bv\in \bH^1(\Ome), \no \\
&\bigl((c_0 p +\alpha\div\bu)_t, \varphi \bigr)_{\rm dual}
+ \frac{1}{\mu_f} \bigl( K(\nab p-\rho_f\bg), \nab \varphi \bigr)
\label{2.33} \\
&\hskip 2in =\bigl(\phi,\varphi\bigr)
+\langle \phi_1,\varphi \rangle
&&\quad\forall \varphi \in H^1(\Ome), \no  \\
&\bu(0) = \bu_0,\qquad p(0)=p_0.  && \label{2.34}
\end{alignat}
\end{definition}
Similarly, we can define the weak solution to the problem \reff{2.21}-\reff{2.23} as follows:
\begin{definition}\label{weak2}
Let $\bu_0\in \bH^1(\Ome), \bbf \in \bL^2(\Omega),
\bbf_1 \in \bL^2(\p\Ome), p_0\in L^2(\Ome), \phi\in L^2(\Ome)$,
and $\phi_1\in L^2(\p\Ome)$.  Assume $c_0>0$ and
$(\bbf,\bv)+\langle \bbf_1, \bv \rangle =0$ for any $\bv\in \mathbf{RM}$.
Given $T > 0$, a $5$-tuple $(\bu,\xi,\eta,p,q)$ with
\begin{alignat*}{2}
&\bu\in L^\infty\bigl(0,T; \bH_\perp^1(\Ome)), &&\qquad
\xi\in L^\infty \bigl(0,T; L^2(\Omega)\bigr), \\
&\eta\in L^\infty\bigl(0,T; L^2(\Omega)\bigr)
\cap H^1\bigl(0,T; H^{1}(\Omega)'\bigr),
&&\qquad q\in L^\infty(0,T;L^2(\Ome)), \\
&p\in L^\infty \bigl(0,T; L^2(\Omega)\bigr) \cap L^2 \bigl(0,T; H^1(\Omega)\bigr)  &&
\end{alignat*}
is called a weak solution to the problem \reff{2.21}-\reff{2.23},
if there hold for almost every $t \in [0,T]$
\begin{alignat}{2}\label{2.35}
\bigl(\mathcal{N}(\nabla\bu), \vepsi(\bv) \bigr)-\bigl( \xi, \div \bv \bigr)
&= (\bbf, \bv)+\langle \bbf_1,\bv\rangle
&&\quad\forall \bv\in \bH^1(\Ome), \\
\kappa_3 \bigl( \xi, \varphi \bigr) +\bigl(\div\bu, \varphi \bigr)
&= \kappa_1\bigl(\eta, \varphi \bigr) &&\quad\forall \varphi \in L^2(\Ome), \label{2.36}  \\
\bigl(\eta_t, \psi \bigr)_{\rm dual}
+\frac{1}{\mu_f} \bigl(K(\nab (\kappa_1\xi +\kappa_2\eta) &-\rho_f\bg), \nab \psi \bigr) \label{2.37} \\
&= (\phi, \psi)+\langle \phi_1,\psi\rangle &&\quad\forall \psi \in H^1(\Ome) , \no  \\
p:=\kappa_1\xi +\kappa_2\eta, \qquad
&q:=\kappa_1\eta-\kappa_3\xi, && \label{2.38} \\
%\bu(0) = \bu_0, \qquad p(0) &=p_0, && \label{e2.8} \\
%q(0)=q_0:=\div \bu_0,\quad \quad
\eta(0)= \eta_0:&=c_0p_0+\alpha q_0,  && \label{2.39}
\end{alignat}
where $q_0:=\div \bu_0$, $u_0$ and $p_0$ are same as in Definition \reff{weak1}.
\end{definition}

\begin{remark}\label{rem-2.1}
It should be pointed out that the only reason for introducing the space $\bH_\perp^1(\Ome)$
in the above two definitions is that the boundary condition \reff{2.16} is a pure ``Neumann condition". If it is replaced by a pure Dirichlet condition or by a mixed Dirichlet-Neumann condition, there is no need
to introduce this space. Thus, from the analysis point of view, the pure Neumann condition case is the most difficult case.
\end{remark}

\section{Existence and uniqueness of weak solution}\label{sec-4}
The proof of next two lemmas about the stress-strain relation are required to obtain a well-posed weak solution of the nonlinear poroelsaticity problem.
\begin{lemma}
	There exist positive constants $ C_{1},~C_{2}$ and $C_{4}$ such that
	\begin{align}
		&\left\|\mathcal{N}(\nabla\mathbf{u}) \right\|_{L^{2}(\Omega_{T})}\leq C_{1}\left\|\varepsilon(\mathbf{u}) \right\|_{L^{2}(\Omega)},\label{2.27}\\ 
		&(\mathcal{N}(\nabla(\mathbf{u})),\varepsilon(\mathbf{u}))\geq C_{2}\left\|\varepsilon(\mathbf{u}) \right\|_{L^{2}(\Omega)}^{2},\label{2.28}\\
		&(\mathcal{N}(\nabla(\mathbf{u}))-\mathcal{N}(\nabla(\mathbf{v})),\varepsilon(\mathbf{u})-\varepsilon(\mathbf{v}))\geq C_{4}\left\|\varepsilon(\mathbf{u})-\varepsilon(\mathbf{v}) \right\|_{L^{2}(\Omega)}^{2}.\label{2.31}
	\end{align}
\end{lemma}
\begin{proof}
	Firstly, we know that  $ \left\| \nabla\mathbf{u}\right\|_{L^{2}(\Omega)} $ is bounded from \cite{201912095}, i.e. $ M\leq\left\|\nabla\mathbf{u} \right\| _{L^{2}\Omega}\leq N $. So we can get $ M^{'}\leq\left\|\nabla\mathbf{u} \right\| _{F}\leq N^{'} $. Using the Cauchy-Schwarz inequality and Korn's inequality, we have
	\begin{align}
		&\left\| \mathcal{N}(\nabla(\mathbf{u}))\right\|_{L^{2}(\Omega)}=\left\|\mu\varepsilon(\mathbf{u})+\mu\nabla^{T}\mathbf{u}\nabla\mathbf{u}+\lambda\left\|\nabla\mathbf{u} \right\|_{F}^{2}\mathbf{I}  \right\|_{L^{2}(\Omega)}\nonumber\\
		&\leq\left\|\mu\varepsilon(\mathbf{u})\right\|_{L^{2}(\Omega)}  +\left\| \mu\nabla\mathbf{u}^{T}\nabla\mathbf{u}\right\|_{L^{2}(\Omega)}  +\left\| \lambda\left\|\nabla\mathbf{u} \right\|_{F}^{2}\mathbf{I}  \right\|_{L^{2}(\Omega)}\nonumber\\
		&\leq\mu\left\|\nabla\mathbf{u}\right\|_{L^{2}(\Omega)}+\mu\left\|\nabla\mathbf{u}\right\|_{L^{2}(\Omega)}^{2}+\lambda\left\|\nabla\mathbf{u} \right\|_{F}^{2}d\nonumber\\
		&\leq(\mu+N\mu+\lambda d \dfrac{N^{'}}{M})\left\|\nabla\mathbf{u}\right\|_{L^{2}(\Omega)}\nonumber\\
		&\leq c_{2}(\mu+N\mu+\lambda d \dfrac{N^{'}}{M})\left\|\varepsilon(\mathbf{u})\right\|_{L^{2}(\Omega)}.\label{eq210823-7}
	\end{align}
	Taking $ C_{1}=c_{2}(\mu+N\mu+\lambda d \dfrac{N^{'}}{M})$ in \reff{eq210823-7}, we see that \reff{2.27}  holds.
	
	To prove \reff{2.28}, we use the fact of $ \dfrac{1}{2}\left[ (a+b,a+b)-(a,a)-(b,b) \right]=(a,b) $ and the inequality $ \left\| x+y\right\|_{L^{2}(\Omega)}^{2}\geq(\left\| x\right\|_{L^{2}(\Omega)}-\left\| y\right\|_{L^{2}(\Omega)})^{2}\geq \left\| x\right\|_{L^{2}(\Omega)}^{2}-\left\| y\right\|_{L^{2}(\Omega)}^{2} $ if and only if $\left\|x \right\|_{L^{2}(\Omega)}\leq \left\|y\right\|_{L^{2}(\Omega)}$ to get
	\begin{align}
		&(\mathcal{N}(\nabla(\mathbf{u})),\varepsilon(\mathbf{u}))=(\mu\varepsilon(\mathbf{u})+\mu\nabla^{T}\mathbf{u}\nabla\mathbf{u}+\lambda\left\|\nabla\mathbf{u} \right\|_{F}^{2}\mathbf{I},\varepsilon(\mathbf{u}))\nonumber\\
		&=(\mu\varepsilon(\mathbf{u}),\varepsilon(\mathbf{u}))+(\mu\nabla^{T}\mathbf{u}\nabla\mathbf{u}+\lambda\left\|\nabla\mathbf{u} \right\|_{F}^{2}\mathbf{I},\varepsilon(\mathbf{u}))\nonumber\\
		&=\mu\left\|\varepsilon(\mathbf{u})\right\|_{L^{2}(\Omega)}^{2}\nonumber\\
		&~+\dfrac{1}{2}\left[ \left(\mu\nabla^{T}\mathbf{u}\nabla\mathbf{u}+\lambda\left\|\nabla\mathbf{u} \right\|_{F}^{2}\mathbf{I}+\varepsilon(\mathbf{u}), \mu\nabla^{T}\mathbf{u}\nabla\mathbf{u}+\lambda\left\|\nabla\mathbf{u} \right\|_{F}^{2}\mathbf{I}+\varepsilon(\mathbf{u})\right)\right.\nonumber\\
		&\left.~-\left(\mu\nabla^{T}\mathbf{u}\nabla\mathbf{u}+\lambda\left\|\nabla\mathbf{u} \right\|_{F}^{2}\mathbf{I},\mu\nabla^{T}\mathbf{u}\nabla\mathbf{u}+\lambda\left\|\nabla\mathbf{u} \right\|_{F}^{2}\mathbf{I} \right)-(\varepsilon(\mathbf{u}),\varepsilon(\mathbf{u}))\right] \nonumber\\
		&=\mu\left\|\varepsilon(\mathbf{u})\right\|_{L^{2}(\Omega)}^{2}+\dfrac{1}{2}\left[ \left\|\mu\nabla^{T}\mathbf{u}\nabla\mathbf{u}+\lambda\left\|\nabla\mathbf{u} \right\|_{F}^{2}\mathbf{I}+\varepsilon(\mathbf{u})\right\|_{L^{2}(\Omega)}^{2}\right.\nonumber\\
		&\left.~-\left\|\mu\nabla^{T}\mathbf{u}\nabla\mathbf{u}+\lambda\left\|\nabla\mathbf{u} \right\|_{F}^{2}\mathbf{I}\right\|_{L^{2}(\Omega)}^{2}-\left\|\varepsilon(\mathbf{u})\right\|_{L^{2}(\Omega)}^{2}\right]\nonumber\\
		&\geq\mu\left\|\varepsilon(\mathbf{u})\right\|_{L^{2}(\Omega)}^{2}+\dfrac{1}{2}\left[\left\|\mu\nabla^{T}\mathbf{u}\nabla\mathbf{u}+\lambda\left\|\nabla\mathbf{u} \right\|_{F}^{2}\mathbf{I}\right\|_{L^{2}(\Omega)}^{2}\right.\nonumber\\
		&~-\left[2c_{2}(\mu N+\dfrac{\lambda dN^{'2}}{M})-1\right] \left\|\varepsilon(\mathbf{u})\right\|_{L^{2}(\Omega)}^{2}\nonumber\\
		&\left.~-\left\|\mu\nabla^{T}\mathbf{u}\nabla\mathbf{u}+\lambda\left\|\nabla\mathbf{u} \right\|_{F}^{2}\mathbf{I}\right\|_{L^{2}(\Omega)}^{2}-\left\|\varepsilon(\mathbf{u})\right\|_{L^{2}(\Omega)}^{2} \right] \nonumber\\
		&=\left(\mu-c_{2}(\mu N+\dfrac{\lambda dN^{'2}}{M}) \right)\left\|\varepsilon(\mathbf{u})\right\|_{L^{2}(\Omega)}^{2}.\label{eq210823-9}
	\end{align}
Due to
	\begin{eqnarray*}
	&\left\|\mu\nabla^{T}\mathbf{u}\nabla\mathbf{u}+\lambda\left\|\nabla\mathbf{u} \right\|_{F}^{2}\mathbf{I}\right\|_{L^{2}(\Omega)}\leq\mu\left\|\nabla\mathbf{u}\right\|_{L^{2}(\Omega)}^{2}+\lambda dN^{'2}\\
	&\leq(\mu N+\dfrac{\lambda dN^{'2}}{M})\left\|\nabla\mathbf{u}\right\|_{L^{2}(\Omega)}\leq c_{2}(\mu N+\dfrac{\lambda dN^{'2}}{M})\left\|\varepsilon(\mathbf{u})\right\|_{L^{2}(\Omega)},
\end{eqnarray*}
and	taking $C_{2}=\mu-c_{2}(\mu N+\dfrac{\lambda dN^{'2}}{M})>0$ in \reff{eq210823-9}, we see that \reff{2.28} holds.

The proof of \reff{2.31} is similar to \reff{2.28}, in fact, we obtain
	\begin{eqnarray}
		&&(\mathcal{N}(\nabla(\mathbf{u}))-\mathcal{N}(\nabla(\mathbf{v})),\varepsilon(\mathbf{u})-\varepsilon(\mathbf{v}))\no\\
		%&&=( \mu\varepsilon(\mathbf{u})+\mu\nabla^{T}\mathbf{u}\nabla\mathbf{u}+\lambda\left\|\nabla\mathbf{u} \right\|_{F}^{2}\mathbf{I}-(\mu\varepsilon(\mathbf{v})+\mu\nabla^{T}\mathbf{v}\nabla\mathbf{v}+\lambda\left\|\nabla\mathbf{v} \right\|_{F}^{2}\mathbf{I}),\no\\
		%&&~\varepsilon(\mathbf{u})-\varepsilon(\mathbf{v})) \no\\
		&&=\mu(\varepsilon(\mathbf{u})-\varepsilon(\mathbf{v}),\varepsilon(\mathbf{u})-\varepsilon(\mathbf{v}))\no\\
		&&~+\left( \mu\nabla^{T}\mathbf{u}\nabla\mathbf{u}-\mu\nabla^{T}\mathbf{v}\nabla\mathbf{v}+\lambda\left\|\nabla\mathbf{u} \right\|_{F}^{2}\mathbf{I}-\lambda\left\|\nabla\mathbf{v} \right\|_{F}^{2}\mathbf{I},\varepsilon(\mathbf{u})-\varepsilon(\mathbf{v})\right) \no\\
		&&=\mu\left\|\varepsilon(\mathbf{u})-\varepsilon(\mathbf{v})\right\|_{L^{2}(\Omega)}^{2}\no\\
		&&~+\dfrac{1}{2}\left[ \left(\mu\nabla^{T}\mathbf{u}\nabla\mathbf{u}-\mu\nabla^{T}\mathbf{v}\nabla\mathbf{v}+\lambda\left\|\nabla\mathbf{u} \right\|_{F}^{2}\mathbf{I}-\lambda\left\|\nabla\mathbf{v} \right\|_{F}^{2}\mathbf{I}+\varepsilon(\mathbf{u})-\varepsilon(\mathbf{v}),\right.\right.\no\\
		&&\left.~\mu\nabla^{T}\mathbf{u}\nabla\mathbf{u}-\mu\nabla^{T}\mathbf{v}\nabla\mathbf{v}+\lambda\left\|\nabla\mathbf{u} \right\|_{F}^{2}\mathbf{I}-\lambda\left\|\nabla\mathbf{v} \right\|_{F}^{2}\mathbf{I}+\varepsilon(\mathbf{u})-\varepsilon(\mathbf{v}) \right)\no\\
		&&~-\left(\mu\nabla^{T}\mathbf{u}\nabla\mathbf{u}-\mu\nabla^{T}\mathbf{v}\nabla\mathbf{v}+\lambda\left\|\nabla\mathbf{u} \right\|_{F}^{2}\mathbf{I}-\lambda\left\|\nabla\mathbf{v} \right\|_{F}^{2}\mathbf{I},\mu\nabla^{T}\mathbf{u}\nabla\mathbf{u}-\mu\nabla^{T}\mathbf{v}\nabla\mathbf{v}\right.\no\\
		&&\left.\left.~+\lambda\left\|\nabla\mathbf{u} \right\|_{F}^{2}\mathbf{I}-\lambda\left\|\nabla\mathbf{v} \right\|_{F}^{2}\mathbf{I} \right)-(\varepsilon(\mathbf{u})-\varepsilon(\mathbf{v}),\varepsilon(\mathbf{u})-\varepsilon(\mathbf{v})) \right] \no\\
		&&=\mu\left\|\varepsilon(\mathbf{u})-\varepsilon(\mathbf{v})\right\|_{L^{2}(\Omega)}^{2}\no\\
		&&~+\dfrac{1}{2}\left[ \left\|\mu\nabla^{T}\mathbf{u}\nabla\mathbf{u}-\mu\nabla^{T}\mathbf{v}\nabla\mathbf{v}+\lambda\left\|\nabla\mathbf{u} \right\|_{F}^{2}\mathbf{I}-\lambda\left\|\nabla\mathbf{v} \right\|_{F}^{2}\mathbf{I}+\varepsilon(\mathbf{u})-\varepsilon(\mathbf{v})\right\|_{L^{2}(\Omega)}^{2}\right.\no\\
		&&~-\left\|\mu\nabla^{T}\mathbf{u}\nabla\mathbf{u}-\mu\nabla^{T}\mathbf{v}\nabla\mathbf{v}+\lambda\left\|\nabla\mathbf{u} \right\|_{F}^{2}\mathbf{I}-\lambda\left\|\nabla\mathbf{v} \right\|_{F}^{2}\mathbf{I}\right\|_{L^{2}(\Omega)}^{2}\no\\
		&&\left.~-\left\|\varepsilon(\mathbf{u})-\varepsilon(\mathbf{v})\right\|_{L^{2}(\Omega)}^{2}\right]\no\\
		&&\geq\mu\left\|\varepsilon(\mathbf{u})-\varepsilon(\mathbf{v})\right\|_{L^{2}(\Omega)}^{2}\no\\
		&&~+\dfrac{1}{2}\left[\left\|\mu\nabla^{T}\mathbf{u}\nabla\mathbf{u}-\mu\nabla^{T}\mathbf{v}\nabla\mathbf{v}+\lambda\left\|\nabla\mathbf{u} \right\|_{F}^{2}\mathbf{I}-\lambda\left\|\nabla\mathbf{v} \right\|_{F}^{2}\mathbf{I}\right\|_{L^{2}(\Omega)}^{2}\right.\no\\
		&&~-\left[2c_{2}(2\mu N+\dfrac{N^{'2}-M^{'2}}{N-M})-1 \right]\left\|\varepsilon(\mathbf{u})-\varepsilon(\mathbf{v})\right\|_{L^{2}(\Omega)}^{2}\no\\
		&&~-\left\|\mu\nabla^{T}\mathbf{u}\nabla\mathbf{u}-\mu\nabla^{T}\mathbf{v}\nabla\mathbf{v}+\lambda\left\|\nabla\mathbf{u} \right\|_{F}^{2}\mathbf{I}-\lambda\left\|\nabla\mathbf{v} \right\|_{F}^{2}\mathbf{I}\right\|_{L^{2}(\Omega)}^{2}\no\\
		&&\left.-\left\|\varepsilon(\mathbf{u})-\varepsilon(\mathbf{v})\right\|_{L^{2}(\Omega)}^{2}\right]\no\\
		&&\qquad=\left(\mu-2c_{2}(2\mu N+\dfrac{N^{'2}-M^{'2}}{N-M}) \right)\left\|\varepsilon(\mathbf{u})-\varepsilon(\mathbf{v})\right\|_{L^{2}(\Omega)}^{2}.\label{eq210823-10}    
	\end{eqnarray}
It is easy to check that
\begin{align*}
	&\left\|\mu\nabla^{T}\mathbf{u}\nabla\mathbf{u}-\mu\nabla^{T}\mathbf{v}\nabla\mathbf{v}+\lambda\left\|\nabla\mathbf{u} \right\|_{F}^{2}\mathbf{I}-\lambda\left\|\nabla\mathbf{v} \right\|_{F}^{2}\mathbf{I}\right\|_{L^{2}(\Omega)}\\
	&\leq\left\|\mu\nabla^{T}\mathbf{u}\nabla\mathbf{u}-\mu\nabla^{T}\mathbf{v}\nabla\mathbf{v}\right\|_{L^{2}(\Omega)}+\left\| \lambda\left\|\nabla\mathbf{u} \right\|_{F}^{2}\mathbf{I}-\lambda\left\|\nabla\mathbf{v} \right\|_{F}^{2}\mathbf{I}\right\|_{L^{2}(\Omega)}\\
	&=\left\|\mu\nabla^{T}\mathbf{u}\nabla\mathbf{u}-\mu\nabla^{T}\mathbf{u}\nabla\mathbf{v}+\mu\nabla^{T}\mathbf{u}\nabla\mathbf{v}-\mu\nabla^{T}\mathbf{v}\nabla\mathbf{v}\right\|_{L^{2}(\Omega)}\no\\
	&~+\left\| \lambda\left\|\nabla\mathbf{u} \right\|_{F}^{2}\mathbf{I}-\lambda\left\|\nabla\mathbf{v} \right\|_{F}^{2}\mathbf{I}\right\|_{L^{2}(\Omega)}\\
	&\leq\mu N\left\|\nabla\mathbf{u}-\nabla\mathbf{v}\right\|_{L^{2}(\Omega)}+\mu N\left\|\nabla\mathbf{u}-\nabla\mathbf{v}\right\|_{L^{2}(\Omega_{T})}+\lambda(N^{'2}-M^{'2})d\\
	&\leq c_{2}(2\mu N+\dfrac{N^{'2}-M^{'2}}{N-M})\left\|\varepsilon(\mathbf{u})-\varepsilon(\mathbf{v})\right\|_{L^{2}(\Omega)}.
\end{align*}
Taking $ C_{4}=\mu-2c_{2}(2\mu N+\dfrac{N^{'2}-M^{'2}}{N-M})>0$, then we get (\ref{2.31}). The proof is complete. 
\end{proof}

\begin{lemma}
	There exists a positive constant $C_{3}$ such that
	\begin{align}
		&\left\|\mathcal{N}(\nabla\mathbf{u})-\mathcal{N}(\nabla\mathbf{v}) \right\|_{L^{2}(\Omega)}\leq C_{3}\left\|\varepsilon(\mathbf{u})-\varepsilon(\mathbf{v}) \right\|_{L^{2}(\Omega)}.\label{2.30} 
	\end{align}
\end{lemma}
\begin{proof}
	Using the Cauchy-Schwarz inequality and Korn's inequality, we have
	\begin{align*}
		&\left\|\mathcal{N}(\nabla\mathbf{u})-\mathcal{N}(\nabla\mathbf{v}) \right\|_{L^{2}(\Omega)}\\
		&=\left\|\mu\nabla^{T}\mathbf{u}\nabla\mathbf{u}-\mu\nabla^{T}\mathbf{v}\nabla\mathbf{v}+\lambda\left\|\nabla\mathbf{u} \right\|_{F}^{2}\mathbf{I}-\lambda\left\|\nabla\mathbf{v} \right\|_{F}^{2}\mathbf{I}+\mu\varepsilon(\mathbf{u})-\mu\varepsilon(\mathbf{v})\right\|_{L^{2}(\Omega)}\\
		&\leq\left\|\mu\nabla^{T}\mathbf{u}\nabla\mathbf{u}-\mu\nabla^{T}\mathbf{v}\nabla\mathbf{v}\right\|_{L^{2}(\Omega)}+\left\| \lambda\left\|\nabla\mathbf{u} \right\|_{F}^{2}\mathbf{I}-\lambda\left\|\nabla\mathbf{v} \right\|_{F}^{2}\mathbf{I}\right\|_{L^{2}(\Omega)}\\
		&~+\left\| \mu\varepsilon(\mathbf{u})-\mu\varepsilon(\mathbf{v})\right\|_{L^{2}(\Omega)}\\
		&=\left\|\mu\nabla^{T}\mathbf{u}\nabla\mathbf{u}-\mu\nabla^{T}\mathbf{u}\nabla\mathbf{v}+\mu\nabla^{T}\mathbf{u}\nabla\mathbf{v}-\mu\nabla^{T}\mathbf{v}\nabla\mathbf{v}\right\|_{L^{2}(\Omega)}\\
		&~+\left\| \lambda\left\|\nabla\mathbf{u} \right\|_{F}^{2}\mathbf{I}-\lambda\left\|\nabla\mathbf{v} \right\|_{F}^{2}\mathbf{I}\right\|_{L^{2}(\Omega)}+\mu\left\|\varepsilon(\mathbf{u})-\varepsilon(\mathbf{v})\right\|_{L^{2}(\Omega)}\\
		&\leq\mu N\left\|\nabla\mathbf{u}-\nabla\mathbf{v}\right\|_{L^{2}(\Omega)}+\mu N\left\|\nabla\mathbf{u}-\nabla\mathbf{v}\right\|_{L^{2}(\Omega)}+\lambda(N^{'2}-M^{'2})d\\
		&~+\mu\left\|\varepsilon(\mathbf{u})-\varepsilon(\mathbf{v})\right\|_{L^{2}(\Omega)}\\
		&=(2\mu N+\dfrac{N^{'2}-M^{'2}}{M-N})\left\|\nabla\mathbf{u}-\nabla\mathbf{v}\right\|_{L^{2}(\Omega)}+\mu\left\|\varepsilon(\mathbf{u})-\varepsilon(\mathbf{v})\right\|_{L^{2}(\Omega)}\\
		&\leq c_{2}(2\mu N+\dfrac{N^{'2}-M^{'2}}{M-N})\left\|\varepsilon(\mathbf{u})-\varepsilon(\mathbf{v})\right\|_{L^{2}(\Omega)}+\mu\left\|\varepsilon(\mathbf{u})-\varepsilon(\mathbf{v})\right\|_{L^{2}(\Omega)}\\
		&=\left(\mu+c_{2}(2\mu N+\dfrac{N^{'2}-M^{'2}}{M-N})\right)\left\|\varepsilon(\mathbf{u})-\varepsilon(\mathbf{v})\right\|_{L^{2}(\Omega)}. 
	\end{align*}
	Taking $ C_{3}=\mu+c_{2}(2\mu N+\dfrac{N^{'2}-M^{'2}}{M-N})>0 $, then \reff{2.30} holds. The proof is complete.
\end{proof}
%%%%%%%
\begin{lemma}\label{lem2.1}
Every weak solution $(\bu,p)$ of the problem \reff{2.32}--\reff{2.34} satisfies
the following energy law:
\begin{align}\label{2.41}
&E(t)+\int_0^t\bigl( \mathcal{N}(\nabla\bu), \varepsilon(\bu_{t})\bigr)\, ds+\frac{1}{\mu_f} \int_0^t \bigl( K(\nab p-\rho_f\bg), \nab p\bigr)\, ds
\\
&\quad-\int_0^t \bigl(\phi, p\bigr)\, ds-\int_0^t \langle \phi_1, p \rangle\, ds =E(0)\no
\end{align}
for all $t\in [0,T]$,  where
\begin{align}\label{2.42}
E(t):&= \frac12 \Bigl[\lam \norm{\div \bu(t)}{L^2(\Ome)}^2
+c_0\norm{p(t)}{L^2(\Ome)}^2-2\bigl(\bbf,\bu(t)\bigr) -2\langle \bbf_1, \bu(t) \rangle \Bigr].
\end{align}
Moreover, there holds
\begin{align}\label{2.43}
\norm{(c_0p+\alpha \div \bu)_t}{L^2(0.T;H^{1}(\Ome)')}
&\leq\frac{1}{\mu_f} \norm{K\nab p-\rho_f \bg}{L^2(\Ome_T)}  \\
&\qquad
+ \|\phi\|_{L^2(\Ome_T)} + \|\phi_1\|_{L^2(\p\Ome_T)} < \infty. \no
%\leq E(0)^{\frac12}.
\end{align}

\end{lemma}

\begin{proof}
We only consider the case of $\bu_t\in \bL^2((0, T);\bL^2(\Omega))$, the general case can be
converted into this case using the Steklov average technique (cf. \cite[Chapter 2]{LSU}).
Setting $\varphi=p$ in \reff{2.33} and $\bv=\bu_t$ in \reff{2.32} yields for a.e. $t\in [0, T]$
\begin{alignat}{2}\label{2.44}
&\mu \bigl(\mathcal{N}(\nabla\bu), \vepsi(\bu_t) \bigr) +\lam\bigl(\div\bu, \div\bu_t \bigr)-\alpha \bigl( p, \div \bu_t \bigr)
=(\bbf, \bu_t)+\langle \bbf_1,\bu_t\rangle,  &&\\
&\bigl((c_0 p +\alpha\div\bu)_t, p(t) \bigr)_{\small\rm dual}
+ \frac{1}{\mu_f} \bigl( K(\nab p-\rho_f\bg), \nab p \bigr)
=\bigl(\phi,p\bigr) +\langle \phi_1, p \rangle.\label{2.5} &&
\end{alignat}
Adding the above two equations and integrating the sum in $t$ over the interval $(0, s)$ for any $s\in(0, T]$, we have
\begin{eqnarray}\label{2.461}
\qquad E(s) + \frac{1}{\mu_f} \int_0^s \bigl(K (\nab p-\rho_f\bg), \nab p\bigr)\, dt
-\int_0^s \bigl(\phi, p\bigr)\, dt
-\int_0^s \langle \phi_1, p \rangle\, dt =E(0).
\end{eqnarray}
Here we have used the fact that $\bbf$ and $\bbf_1$
are independent of $t$. Hence, \reff{2.41} holds. \reff{2.43} follows immediately from \reff{2.41} and \reff{2.33}. The proof is complete.
\end{proof}

Likewise, the weak solution of \reff{2.35}--\reff{2.39} satisfy a similar energy law which is a rewritten version of \reff{2.41} in the new variables.

\begin{lemma}\label{lem2.2}
Every weak solution $(\bu,\xi,\eta)$ of the problem \reff{2.35}--\reff{2.39} satisfies
the following energy law
\begin{align}\label{2.47}
&J(t)+\int_0^t\bigl( \mathcal{N}(\nabla\bu), \varepsilon(\bu_{t})\bigr)\, ds + \frac{1}{\mu_f} \int_0^t \bigl(K(\nab p-\rho_f\bg), \nab p\bigr)\, ds
\\
&\quad-\int_0^t \bigl(\phi, p\bigr)\, ds-\int_0^t \langle \phi_1, p \rangle\, ds =J(0)\no
\end{align}
for all $t\in [0, T]$,  where
\begin{align}\label{2.48}
J(t):&= \frac12 \Bigl[\kappa_2 \norm{\eta(t)}{L^2(\Ome)}^2 +\kappa_3 \norm{\xi(t)}{L^2(\Ome)}^2-2\bigl(\bbf,\bu(t)\bigr) -2\langle \bbf_1, \bu(t) \rangle \Bigr].
\end{align}

Moreover, there holds
\begin{align}\label{2.49}
\norm{\eta_t}{L^2(0.T;H^{1}(\Ome)')}
&\leq\frac{1}{\mu_f} \norm{K\nab p-\rho_f \bg}{L^2(\Ome_T)}  \\
&\qquad
+ \|\phi\|_{L^2(\Ome_T)} + \|\phi_1\|_{L^2(\p\Ome_T)} < \infty. \no
\end{align}

\end{lemma}
\begin{proof}
We only consider the case of $\bu_t\in L^2(0,T; L^2(\Ome))$. Setting $\bv=\bu_t$ in \reff{2.35},
differentiating \reff{2.36} with respect to $t$ followed by taking $\varphi=\xi$, and setting
$\psi=p=\kappa_1\xi +\kappa_2\eta$ in \reff{2.37}, we have
\begin{align}
	\bigl(\mathcal{N}(\nabla\bu), \vepsi(\bu_{t}) \bigr)-\bigl( \xi, \div \bu_{t} \bigr)
	&= (\bbf, \bu_{t})+\langle \bbf_1,\bu_{t}\rangle
	&&\quad\forall \bv\in \bH^1(\Ome), \\
	\kappa_3 \bigl( \xi_{t}, \xi \bigr) +\bigl(\div\bu_{t}, \xi \bigr)
	&= \kappa_1\bigl(\eta_{t}, \xi \bigr) &&\quad\forall \xi \in L^2(\Ome),   \\
	\bigl(\eta_t, p \bigr)_{\rm dual}
	+\frac{1}{\mu_f} \bigl(K(\nab (\kappa_1\xi +\kappa_2\eta) &-\rho_f\bg), \nab p \bigr)  \\
	&= (\phi, p)+\langle \phi_1,p\rangle &&\quad\forall \psi \in H^1(\Ome) . \no  
	\end{align}
 
 Adding the resulting equations and integrating in
$t$, we see that \reff{2.47} holds. The inequality \reff{2.49} follows immediately from
\reff{2.38} and \reff{2.47}. The proof is complete.
\end{proof}

The above energy law immediately implies the following solution estimates.

\begin{lemma}\label{estimates}
There exists a positive constant
$ \acute{C}_1=\acute{C}_1\bigl(\|\bu_0\|_{H^1(\Ome)}, \|p_0\|_{L^2(\Ome)},$
$\|\bbf\|_{L^2(\Ome)},\|\bbf_1\|_{L^2(\p \Ome)},\|\phi\|_{L^2(\Ome)}, \|\phi_1\|_{L^2(\p\Ome)} \bigr)$
such that
\begin{align}\label{2.53}
&\sqrt{C_{2}}\|\varepsilon(\bu)\|_{L^\infty(0,T;L^2(\Ome))}
+\sqrt{\kappa_2} \|\eta\|_{L^\infty(0,T;L^2(\Ome))} \\
&\qquad
+\sqrt{\kappa_3} \|\xi\|_{L^\infty(0,T;L^2(\Ome))}
+\sqrt{\frac{K_1}{\mu_f}} \|\nab p \|_{L^2(0,T;L^2(\Ome))} \leq \acute{C}_1, \no \\
&\|\bu\|_{L^\infty(0,T;L^2(\Ome))}\leq \acute{C}_1, \quad
\|p\|_{L^\infty(0,T;L^2(\Ome))} \leq \acute{C}_1 \bigl( \kappa_2^{\frac12} + \kappa_1 \kappa_3^{-\frac12}
\bigr), \label{2.54} \\
&\|p\|_{L^2(0,T; L^2(\Ome))} \leq \acute{C}_1,~ \quad
\|\xi\|_{L^2(0,T;L^2(\Ome))} \leq \acute{C}_1\kappa_1^{-1} \bigl(1+ \kappa_2^{\frac12} \bigr).
\label{2.55}
\end{align}
\end{lemma}

\begin{proof}
Taking $\psi=p$ in \reff{2.37} and integrating from $0$ to $t$, we have
\begin{align}\label{2.56}
 &\int_0^t\bigl(\eta_t, p \bigr)_{\rm dual}ds
+\int_0^t\frac{1}{\mu_f} \bigl(K(\nab (\kappa_1\xi +\kappa_2\eta) -\rho_f\bg), \nab p \bigr)ds\\
&= \int_0^t[(\phi, p)+\langle \phi_1,p\rangle]ds.\no
\end{align}
Taking $\bv=\bu_{t}$ in \reff{2.35} and $\varphi=\xi$ in\reff{2.36}, we have
\begin{align}\label{2.57}
 \int_0^t\bigl(\eta_t, p \bigr)_{\rm dual}ds &=\int_0^t\bigl(\eta_t, \kappa_1\xi +\kappa_2\eta \bigr)_{\rm dual}ds \\
   &=\int_0^t\bigl(\eta_t, \kappa_1\xi\bigr)_{\rm dual}ds +\int_0^t\bigl(\eta_t, \kappa_2\eta \bigr)_{\rm dual}ds\no\\
   &=\frac{1}{2}\kappa_2 (\norm{\eta(t)}{L^2(\Ome)}^2-\norm{\eta(0)}{L^2(\Ome)}^2)+ \int_0^t\bigl(\eta_t, \kappa_1\xi\bigr)_{\rm dual}ds,\no
\end{align}
\begin{eqnarray}\label{2.58}
 &&\int_0^t\bigl(\eta_t, \kappa_1\xi\bigr)_{\rm dual}ds= \int_0^t\bigl(q_{t}+\kappa_3\xi_{t}, \xi\bigr)_{\rm dual}ds \\
  &&= \int_0^t\bigl(\div\bu_{t}, \xi\bigr)_{\rm dual}ds+\frac{1}{2}\kappa_3\bigl( \norm{\xi(t)}{L^2(\Ome)}^2-\norm{\xi(0)}{L^2(\Ome)}^2\bigr)\no\\
   &&=\int_0^t[\bigl(\mathcal{N}(\nabla\bu), \vepsi(\bu_{t}) \bigr)- (\bbf, \bu_{t})-\langle \bbf_1,\bu_{t}\rangle]ds\no\\
   &&~+\frac{1}{2}\kappa_3\bigl( \norm{\xi(t)}{L^2(\Ome)}^2-\norm{\xi(0)}{L^2(\Ome)}^2\bigr).\no
\end{eqnarray}
Substituting \reff{2.57} and \reff{2.58} into \reff{2.56}, we get
\begin{align*}
&\int_0^t\bigl(\mathcal{N}(\nabla\bu), \vepsi(\bu_{t}) \bigr)ds + \frac12 \Bigl[\kappa_2 \norm{\eta(t)}{L^2(\Ome)}^2 +\kappa_3 \norm{\xi(t)}{L^2(\Ome)}^2]\\
&~+\frac{1}{\mu_f} \int_0^t \bigl(K(\nab p-\rho_f\bg), \nab p\bigr)\, ds \\
&= \frac12 \Bigl[\kappa_2 \norm{\eta(0)}{L^2(\Ome)}^2 +\kappa_3 \norm{\xi(0)}{L^2(\Ome)}^2 +2\bigl(\bbf,\bu(t)-\bu(0)\bigr) +2\langle \bbf_1, \bu(t)-\bu(0) \rangle \Bigr]\\
&~+\int_0^t \bigl(\phi, p\bigr)\, ds+\int_0^t \langle \phi_1, p \rangle\, ds.
\end{align*}
Using \reff{2.28}, we have 
\begin{align*}
	&C_{2}\int_0^t\bigl( \vepsi(\bu, \vepsi(\bu_{t}) \bigr)ds + \frac12 \Bigl[\kappa_2 \norm{\eta(t)}{L^2(\Ome)}^2 +\kappa_3 \norm{\xi(t)}{L^2(\Ome)}^2]\\
	&~+\frac{1}{\mu_f} \int_0^t \bigl(K(\nab p-\rho_f\bg), \nab p\bigr)\, ds \\
	&\leq\frac12 \Bigl[\kappa_2 \norm{\eta(0)}{L^2(\Ome)}^2 +\kappa_3 \norm{\xi(0)}{L^2(\Ome)}^2 +2\bigl(\bbf,\bu(t)-\bu(0)\bigr) +2\langle \bbf_1, \bu{t}-\bu(0) \rangle \Bigr]\\
	&~+\int_0^t \bigl(\phi, p\bigr)\, ds+\int_0^t \langle \phi_1, p \rangle\, ds.
\end{align*}
%$$\bigl(\mathcal{N}(\nabla\bu), \vepsi(\bu_{t}) \bigr)\geq\delta\|\varepsilon(\mathbf{u})\|_{L^{2}}^{2}-\tau\|\nabla\mathbf{u}_{t}\|_{L^{2}}^{4}$$
Hence, \reff{2.53} holds. It's easy to check that \reff{2.54} holds from \reff{2.53} and the relation $p=\kappa_1\xi +\kappa_2\eta$. We note that \reff{2.55} follows from \reff{2.53}, \reff{eq210823-11}, the Poincar$\acute{e}$ inequality and \reff{2.70} below, and the relation $p=\kappa_1\xi +\kappa_2\eta$. The proof is complete.
\end{proof}
\begin{theorem}\label{smooth}
Suppose that $\bu_0$ and $p_0$ are sufficiently smooth, then
there exist positive constants~ $ \acute{C}_2=\acute{C}_2\bigl(\acute{C}_1,\|\nab p_0\|_{L^2(\Ome)} \bigr)$
and~ $ \acute{C}_3=\acute{C}_3\bigl(\acute{C}_1,\acute{C}_2, \|\bu_0\|_{H^2(\Ome)},\|p_0\|_{H^2(\Ome)} \bigr)$ such that
\begin{align}\label{2.59}
&\sqrt{C_{2}}\|\varepsilon(\bu_{t})\|_{L^2(0,T;L^2(\Ome))}
+\sqrt{\kappa_2} \|\eta_t\|_{L^2(0,T;L^2(\Ome))} \\
&\qquad
+\sqrt{\kappa_3} \|\xi_t\|_{L^2(0,T;L^2(\Ome))}
+\sqrt{\frac{K_1}{\mu_f}} \|\nab p \|_{L^\infty(0,T;L^2(\Ome))} \leq \acute{C}_2, \no \\
&\sqrt{C_{2}}\|\varepsilon(\bu_{t})\|_{L^\infty(0,T;L^2(\Ome))}
+\sqrt{\kappa_2} \|\eta_t\|_{L^\infty(0,T;L^2(\Ome))}\label{2.60} \\
&\qquad
+\sqrt{\kappa_3} \|\xi_t\|_{L^\infty(0,T;L^2(\Ome))}
+\sqrt{\frac{K_1}{\mu_f}} \|\nab p_t \|_{L^2(0,T;L^2(\Ome))} \leq \acute{C}_3, \no \\
&\|\eta_{tt}\|_{L^2(H^{1}(\Ome)')} \leq \sqrt{\frac{K_2}{\mu_f}}\acute{C}_3. \label{2.61}
\end{align}
\end{theorem}

\begin{proof}
Differentiating \reff{2.35} and \reff{2.36} with respect to $t$, taking $\bv=\bu_t$
and $\varphi=\xi_t$ in \reff{2.35} and \reff{2.36} respectively, and adding the
resulting equations, we have 
\begin{align}\label{2.62}
\bigl(\mathcal{N}_{t}(\nabla\bu),\varepsilon(\bu_{t})\bigr) = \bigl(q_t,\xi_t \bigr)
=\kappa_1 \bigl(\eta_t,\xi_t\bigr) -\kappa_3\|\xi_t\|_{L^2(\Ome)}^2.
\end{align}
Setting $\psi=p_t=\kappa_1\xi_t + \kappa_2\eta_t$ in \reff{2.37}, we get 
\begin{align}\label{2.63}
\kappa_1 \bigl(\eta_t,\xi_t \bigr) + \kappa_2\|\eta_t\|_{L^2(\Ome)}^2
+\frac{K}{2\mu_f} \frac{d}{dt} \|\nab p-\rho_f\bg\|_{L^2(\Ome)}^2
=\frac{d}{dt}\Bigl[ (\phi,p) +\langle \phi_1, p\rangle \Bigr].
\end{align}
Adding \reff{2.62} and \reff{2.63} and integrating in $t$ we get for $t\in[0,T]$, we have
\begin{align*}%\label{eq3.23c}
&\frac{K}{2\mu_f} \|\nab p(t)-\rho_f\bg\|_{L^2(\Ome)}^2
+\int_0^t \Bigl[ \bigl(\mathcal{N}_{t}(\nabla\bu),\varepsilon(\bu_{t})\bigr)
+ \kappa_2\|\eta_t\|_{L^2(\Ome)}^2 + \kappa_3\|\xi_t\|_{L^2(\Ome)}^2\Bigr]\,ds \\
&\hskip 0.65in
=\frac{K}{2\mu_f} \|\nab p_0-\rho_f\bg\|_{L^2(\Ome)}^2
+ (\phi,p(t)-p_0) +\langle \phi_1, p(t)-p_0 \rangle, \no
\end{align*}
%$$ \bigl(\mathcal{N}_{t}(\nabla\bu),\varepsilon(\bu_{t})\bigr)\geq \delta\|\varepsilon(\bu_{t})\|_{L^2(\Ome)}^2-2\mu\|\nabla\bu_{t}\|_{L^2(\Ome)}^2\|\nabla\bu\|_{L^2(\Ome)}^2-\frac{\lambda}{4}\|\nabla\bu_{t}\|_{L^2(\Ome)}^4\|\nabla\bu\|_{L^2(\Ome)}^4$$
which implies that \reff{2.59} holds. 

To show \reff{2.60}, first differentiating \reff{2.35} one time with respect to $t$ and setting
$\bv=\bu_{tt}$, differentiating \reff{2.36} twice with respect to $t$ and setting
$\varphi=\xi_t$, and adding the resulting equations, we get
\begin{align}\label{2.64}
\bigl(\mathcal{N}_{t}(\nabla\bu),\varepsilon(\bu_{tt})\bigr) = \bigl(q_{tt},\xi_t \bigr)
=\kappa_1\bigl(\eta_{tt},\xi_t\bigr) -\frac{\kappa_3}{2}\frac{d}{dt}\|\xi_t\|_{L^2(\Ome)}^2.
\end{align}

Secondly, differentiating \reff{2.37} with respect $t$ one time and taking
$\psi=p_t=\kappa_1\xi_t + \kappa_2\eta_t$, we get
\begin{align}\label{2.65}
\kappa_1 \bigl(\eta_{tt},\xi_t \bigr) + \frac{\kappa_2}{2} \frac{d}{dt} \|\eta_t\|_{L^2(\Ome)}^2
+\frac{K}{\mu_f} \|\nab p_t\|_{L^2(\Ome)}^2 =0.
\end{align}
Finally, adding \reff{2.64}-\reff{2.65} and integrating in $t$, we obtain
\begin{align}\label{2.66}
&2\int_0^t\bigl(\mathcal{N}_{t}(\nabla\bu),\varepsilon(\bu_{tt})\bigr)\,ds +\kappa_2 \|\eta_t(t)\|_{L^2(\Ome)}^2
+\kappa_3 \|\xi_t(t)\|_{L^2(\Ome)}^2 \\
&\hskip 0.6in+ \frac{2K}{\mu_f} \int_0^t \|\nab p_t\|_{L^2(\Ome)}^2\,ds
=\kappa_2 \|\eta_t(0)\|_{L^2(\Ome)}^2
+\kappa_3 \|\xi_t(0)\|_{L^2(\Ome)}^2,\no
\end{align}
%where
%\begin{align*}
%    2\int_0^t\bigl(\mathcal{N}_{t}(\nabla\bu),\varepsilon(\bu_{tt})\bigr)\,ds&= 2\delta\int_0^t\bigl(\mathcal{N}_{t}(\nabla\bu),\varepsilon(\bu_{tt})\bigr)\,ds\\
%&\geq \delta\|\varepsilon(\bu_{t}(t))\|_{L^2(\Ome)}^2-\delta\|\varepsilon(\bu_{t}(0))\|_{L^2(\Ome)}^2\\
%&-2\int_0^t[2\mu\|\nabla\bu_{t}\|_{L^2(\Ome)}^2\|\nabla\bu_{tt}\|_{L^2(\Ome)}^2+\frac{\lambda}{4}\|\nabla\bu_{t}\|_{L^2(\Ome)}^4\|\nabla\bu_{tt}\|_{L^2(\Ome)}^4]\,ds
%\end{align*}
which implies that \reff{2.60} holds. \reff{2.61} follows immediately from the following inequality
\begin{align*}
\bigl( \eta_{tt}, \psi \bigr) = -\frac{1}{\mu_f} \bigl( K\nab p_{t}, \nab \psi\bigr)
\leq \frac{K}{\mu_f} \|\nab p_{t}\|_{L^2(\Ome)} \|\nab \psi\|_{L^2(\Ome)},
\end{align*}
\reff{2.60} and the definition of the $H^{1}(\Omega)'$-norm. The proof is complete.
\end{proof}
\begin{remark}
The above estimates require $p_0\in H^1(\Ome),~
\bu_t(0)\in \bL^2(\Ome),~ \eta_t(0)\in L^2(\Ome)$ and $\xi_t(0)\in L^2(\Ome)$. The values
of $\bu_t(0),~ \eta_t(0)$ and $\xi_t(0)$ can be computed using the PDEs as follows. It follows from \reff{2.23} that $\eta_t(0)$ satisfies
\begin{align*}
\eta_t(0)= \phi + \frac{1}{\mu_f} \div[K (\nab p_0-\rho_f\bg)].
\end{align*}
Hence, $\eta_t(0)\in L^2(\Ome)$ provided that $p_0\in H^2(\Ome)$. To find $\bu_t(0)$ and $\xi_t(0)$, differentiating \reff{2.21} and \reff{2.22} with respect to $t$ and setting $t=0$, we get
\begin{alignat*}{2}
-\div\mathcal{N}\bigl(\nabla\bu_t(0)\bigr) + \nab \xi_t(0) &=0 &&\qquad \mbox{\rm in } \Ome,\\
\kappa_3\xi_t(0) +\div \bu_t(0) &=\kappa_1\eta_t(0)  &&\qquad \mbox{\rm in } \Ome.
\end{alignat*}
Hence, $\bu_t(0)$ and $\xi_t(0)$ can be determined by solving the above generalized Stokes problem.
\end{remark}

The next lemma shows that the weak solution of the problem \reff{2.35}-\reff{2.39} preserves some ``invariant" quantities, it turns out that these ``invariant" quantities play a vital role in the proof of existence and uniqueness of the weak solution to the reformulated fluid-fluid coupling system.

\begin{lemma}\label{lem2.9}
Every weak solution $(\bu,\xi,\eta,p, q)$ to the problem \reff{2.35}-\reff{2.39} satisfies
the following energy laws
\begin{align}\label{2.67}
&C_\eta(t):=\bigl(\eta(\cdot, t),1\bigr)
=\bigl(\eta_0,1\bigr) + \bigl[ (\phi,1) + \langle \phi_1, 1\rangle \bigr] t,\quad t\geq 0,\\
&C_\xi(t):=\bigl( \xi(\cdot,t), 1\bigr),
\label{2.68} \\
&C_q(t):=\bigl( q(\cdot,t), 1\bigr)
=\kappa_1 C_\eta(t)-\kappa_3 C_\xi(t), \label{2.69}\\
&C_p(t):=\bigl( p(\cdot,t), 1\bigr)
=\kappa_1 C_\xi(t)+\kappa_2 C_\eta(t), \label{2.70}\\
&C_\bu(t) := \bigl\langle \bu(\cdot,t)\cdot \bn, 1 \bigr\rangle  =C_q(t). \label{2.71}
\end{align}
\end{lemma}
\begin{proof}
We first notice that \reff{2.67} follows immediately from taking $\psi\equiv 1$
in \reff{2.37}. To prove \reff{2.68}, taking $\bv=\bx$ in \reff{2.35} and $\varphi=1$ in \reff{2.36}, which are
valid test functions, and using the identities $\nabla \bx=\mathbf{I},~ \div \bx=d$, and $\varepsilon(\bx)=\mathbf{I}$, we get
\begin{align*}
\Bigl(\mathcal{N}\bigl(\nabla\bu\bigr), \mathbf{I}\Bigr) &= d\bigl( \xi, 1\bigr)+\bigl( \bbf, \bx\bigr)+ \langle \bbf_1,\bx\rangle,\\
\bigl( \div \bu, 1\bigr) &=\kappa_1(\eta,1)-\kappa_3(\xi, 1).
\end{align*}
It is easy to check that
$$C_\xi(t):=\bigl( \xi(\cdot,t), 1\bigr)=\frac{1}{d-\kappa_3}\bigl[\Bigl(\mathcal{N}\bigl(\nabla\bu\bigr), \mathbf{I}\Bigr)+\bigl( \div \bu, 1\bigr)-\kappa_1C_\eta(t)-\bigl(\bbf, \bx\bigr)- \langle \bbf_1,\bx\rangle \bigr],$$
which implies that \reff{2.68} holds.

Finally, since $q=\kappa_1\eta-\kappa_3\xi$,$p=\kappa_1\xi + \kappa_2 \eta$, \reff{2.69} and \reff{2.70} follow from \reff{2.67}
and \reff{2.68}. \reff{2.71} is an immediate consequence of $q=\div \bu$ and the Gauss divergence theorem. The proof is complete.
\end{proof}
%\begin{remark}
%{\color{red}We note that $C_\eta, C_\xi, C_q$ and $C_p$ all are (known) linear functions of $t$},
%and they become (known) constants when $\phi\equiv 0$ and $\phi_1\equiv 0$.
%\end{remark}

With the help of the above lemmas, we can show the solvability of the problem \reff{2.6}-\reff{2.18}.

\begin{theorem}\label{thm2.5}
Let $\bu_0\in\bH^1(\Ome), \bbf\in\bL^2(\Omega),
\bbf_1\in \bL^2(\p\Ome), p_0\in L^2(\Ome), \phi\in L^2(\Ome)$, and $\phi_1\in L^2(\p\Ome)$. Suppose $c_0>0$ and $(\bbf,\bv)+\langle \bbf_1, \bv \rangle =0$ for any $\bv\in \mathbf{RM}$. Then there exists a unique weak solution to the problem \reff{2.6}-\reff{2.18} in the sense of Definition \ref{weak1}. Likewise, there exists a unique weak solution to the problem
\reff{2.21}-\reff{2.26} in the sense of Definition \ref{weak2}.
\end{theorem}
\begin{proof}
	We first prove the existence of solution of the problem \reff{2.35}-\reff{2.39}. Given a function $ \bu\in L^\infty\bigl(0,T; \bH_\perp^1(\Ome))$, supposing that $U\subset L^\infty\bigl(0,T; \bH_\perp^1(\Ome)) $ is a compact and convex subspace, defining $ g(t):=-\mu\nabla^{T}\bu\nabla\bu-\lambda\left\|\nabla\bu\right\|_{F}^{2}\mathbf{I} (0\leq t\leq T)$, we have
	\begin{align}\label{6.9}
		\left\|\mu\nabla^{T}\mathbf{u}\nabla\mathbf{u}+\lambda\left\|\nabla\mathbf{u} \right\|_{F}^{2}\mathbf{I}\right\|_{L^{2}(\Ome)}&\leq\mu\left\|\nabla\mathbf{u}\right\|_{L^{2}(\Ome)}^{2}+\lambda dN^{'2}\\
		&\leq(\mu N+\dfrac{\lambda dN^{'2}}{M})\left\|\nabla\mathbf{u}\right\|_{L^{2}(\Ome)}.\no
	\end{align}
	In the sight of \reff{6.9}, we see that $ g\in L^{2}(0,T;L^{2}(\Ome)) $. Denote $\bV=L^\infty\bigl(0,T; \bH_\perp^1(\Ome)),~\Upsilon=L^\infty\bigl(0,T; L^{2}(\Ome)),~W=L^\infty\bigl(0,T; L^2(\Omega)\bigr)
	\cap H^1\bigl(0,T; H^{1}(\Omega)'\bigr) $. We consider the linear problem: find $ (w,\xi,\eta)\in\bV\times \Upsilon\times W $ satisfying 
	\begin{align}
		&\mu\bigl(\vepsi(\bw), \vepsi(\bv) \bigr)-\bigl( \xi, \div \bv\bigr)
		= (g, \bv)+(\bbf,\bv)+\langle \bbf_1,\bv\rangle &&\quad \forall \bv\in \bH^{1}(\Ome), \label{6.10}\\
		&\kappa_3\bigl(\xi, \varphi \bigr) +\bigl(\div\bw, \varphi \bigr)
		=\kappa_1\bigl( \eta, \varphi \bigr)
		&&\quad \forall \varphi \in L^{2}(\Ome), \label{6.11}\\
		&\bigl(d_t\eta, \psi \bigr)
		+\frac{1}{\mu_f} \bigl(K(\nab (\kappa_1\xi +\kappa_2\eta) &&\label{6.12}\\
		&\hskip 1in
		-\rho_f\bg,\nab\psi \bigr)=(\phi, \psi)+\langle \phi_1,\psi\rangle, &&\quad  \forall \psi\in H^{1}(\Ome).\no  
	\end{align}
As for the equations of \reff{2.22}-\reff{2.23}, according to the theory of linear parabolic equations, we know that $ \xi $ and $ \eta $ can be uniquely determined by $ \bw $, that is, $ \exists $ $ \varPhi $ and $ \varPsi $, s.t. $ \xi=\varPhi(\bw) $, $ \eta=\varPsi(\bw) $.  Thus, the problem \reff{6.10}-\reff{6.12} is equivalent to the following problem
	\begin{align}\label{eq210826-1}
		\left\lbrace \begin{aligned}	
	    &Solve~ \bw\in\bV  ~such~ that\\
		&\mu\bigl(\vepsi(\bw), \vepsi(\bv) \bigr)+\bigl( \varPhi(\bw), \div \bv \bigr)
		= (g, \bv)+(\bbf,\bv)+\langle \bbf_1,\bv\rangle ~~~ \forall \bv\in \bH^{1}(\Ome).\\
	\end{aligned}\right.
	\end{align} 
	Following the method of \cite{fh10}, we can prove that the solution of \reff{eq210826-1} uniquely exists, here we omit the details of proof.\\
	Define $ A:\boldsymbol{V}\rightarrow \boldsymbol{V} $ by $ A[\bu]=\bw$. Similarly, it's easy to know the following problem is equivalent to the problem \reff{2.35}-\reff{2.37}.
\begin{align}\label{eq210826-2}
	\left\lbrace \begin{aligned}
	    &Solve ~ \bu\in\bV  ~such~ that\\
		&\bigl(\mathcal{N}(\nabla\bu), \vepsi(\bv) \bigr)+\bigl(\varPhi(\bu), \div \bv \bigr)
		= (\bbf, \bv)+\langle \bbf_1,\bv\rangle &&\quad \forall \bv\in \bH^{1}(\Ome). \\
	\end{aligned}\right.
	\end{align}

	Next, we prove that $ A $ is continuous. To do that, choose $\bu,~\tilde{\bu} $ and define $ \bw=A[\bu],~ \tilde{\bw}=A[\tilde{\bu}] $ as above. Consequently $ \bw $ verifies \reff{6.10}-\reff{6.12} and $ \tilde{\bw} $ satisfies a similar identity for $ \tilde{g}= -\mu\nabla^{T}\tilde{\bu}\nabla\tilde{\bu}-\lambda\left\|\nabla\tilde{\bu}\right\|^{2}_{F}~ \mathbf{I}$.
	
	Using \reff{eq210823-11}, \reff{6.10} and the Young inequality, we have
	\begin{align*}
		&\mu\left\| \tilde{\bw}-\bw\right\|_{L^{2}(\Ome)}^{2}+c_{1}^{2}(\varPhi(\tilde{\bw})-\varPhi(\bw),\tilde{\bw}-\bw)\\
		&\leq c_{1}^{2}\mu\left\| \vepsi(\tilde{\bw})-\vepsi(\bw)\right\|_{L^{2}(\Ome)}^{2}+c_{1}^{2}(\varPhi(\tilde{\bw})-\varPhi(\bw),\tilde{\bw}-\bw)\\
		&=c_{1}^{2}(\tilde{g}-g,\tilde{\bw}-\bw)\\
		&\leq c_{1}^{2}\left[ \epsilon\left\| \tilde{\bw}-\bw\right\|_{L^{2}(\Ome)}^{2}+\dfrac{1}{\epsilon}\left\| \tilde{g}-g\right\|_{L^{2}(\Ome)}^{2}\right] 
	\end{align*}
	by Poincar$\acute{e}$ inequality, where $c_{1}$ is a real positive constant in \reff{eq210823-11}. Selecting $\epsilon>0$ sufficiently small, we have
	\begin{align*}
		c_{1}^{2}(\varPhi(\tilde{\bw})-\varPhi(\bw),\tilde{\bw}-\bw)\leq C_{f}\left\| \tilde{g}-g\right\|_{L^{2}(\Ome)}^{2}\leq C_{f}\left\| \tilde{\bu}-\bu\right\|_{L^{2}(\Ome)}^{2}.
	\end{align*}
	where $C_{f}$ is a real positive number. We discover
	\begin{align*}
		\left\| A[\tilde{\bu}]-A[\bu]\right\|_{L^{2}(\Ome)}^{2}=\left\| \tilde{\bw}-\bw\right\|_{L^{2}(\Ome)}^{2}\leq \tilde{C}_{f}\left\| \tilde{\bu}-\bu\right\|_{L^{2}(\Ome)}^{2}.
	\end{align*}
	Thus, we get
	\begin{align*}
		\left\| A[\tilde{\bu}]-A[\bu]\right\|_{L^{2}(\Ome)}\leq\sqrt{\tilde{C}_{f}}\left\| \tilde{\bu}-\bu\right\|_{L^{2}(\Ome)}.
	\end{align*}
	If $ \tilde{C}_{f} $ is so small, then $ A $ is continuous. Since $U$ is a compact and convex subspace, according to Schauder's fixed point theorem, $A$ has a fixed point in $U$.
	
	Then, we prove that the problem \reff{2.35}-\reff{2.39} has a unique solution. Lemma \ref{estimates} and Theorem \ref{smooth} gives  the priori estimates for the weak solution. Since $ C_{\eta}(t)=(\eta(\cdot,t),1)=(\eta_{0},1)+\left[(\phi,1)+\left\langle \phi_{1},1\right\rangle  \right] $. It's easy to check that $ \eta $ is unique. We assume that $ (\boldsymbol{u}_{1},\xi_{1},\eta_{1},p_{1},q_{1}) $ and $ (\boldsymbol{u}_{2},\xi_{2},\eta_{2},p_{2},q_{2}) $ are the different solutions of \reff{2.35}-\reff{2.39}. Using \reff{2.35} and \reff{2.36}, we obtain
\begin{align}
	(\mathcal{N}(\nabla\bu_{1})-\mathcal{N}(\nabla\bu_{2}),\varepsilon(\bv))-(\xi_{1}-\xi_{2},\nabla\cdot\bv)&=0~~~~\forall~\bv_{h}\in H^{1}(\varOmega),\label{2.72}\\
	\kappa_{3}(\xi_{1}-\xi_{2},\varphi)+(\nabla\cdot\bu_{1}-\nabla\cdot\bu_{2},\varphi)&=0~~~~~\forall \varphi_{h}\in L^{2}(\varOmega).\label{2.73}
\end{align}
Adding \reff{2.72} and \reff{2.73}, letting $ \bv=\bu_{1}-\bu_{2},~ \varphi=\xi_{1}-\xi_{2} $, using \reff{2.30}, we have
\begin{align}
	0\leq C_{3}\left\|\varepsilon(\bu)-\varepsilon(\bv) \right\|^{2}_{L^{2}(\varOmega(t))}+\kappa_{3}\left\|\xi_{1}-\xi_{2} \right\|^{2}_{L^{2}(\varOmega(t))} =0.\label{2.74} 
\end{align}
Thus, using \reff{2.74} and the initial value $ \bu_{0} $, we obtain
\begin{align*}
	\bu_{1}=\bu_{2},~~~\xi_{1}=\xi_{2}.	
\end{align*}
Since $ p=\kappa_{1}\xi+\kappa_{2}\eta ,~q=\kappa_{1}\eta-\kappa_{3}\xi$, so we have
\begin{align*}
	p_{1}=p_{2},~~~~q_{1}=q_{2}.
\end{align*}	
Hence, the solution of the problem \reff{2.35}-\reff{2.39} is unique. The proof is complete.
\end{proof}

We conclude this section by establishing a convergence result for the solution of the
problem \reff{2.21}-\reff{2.23}, \reff{2.24}-\reff{2.26} when the constrained specific storage coefficient $c_0$ tends to $0$. Such a convergence result is useful and significant for that the poroelasticity model
studied in this paper reduces into the nonlinear Biot's consolidation model from soil mechanics \cite{murad,pw07} and some polymer gels \cite{yd04b,fh10}.

\begin{theorem}\label{thm2.6}
Let $\bu_0\in\bH^1(\Ome), \bbf\in\bL^2(\Omega),
\bbf_1\in \bL^2(\p\Ome), p_0\in L^2(\Ome), \phi\in L^2(\Ome)$, and $\phi_1\in L^2(\p\Ome)$. Suppose $(\bbf,\bv)+\langle \bbf_1, \bv \rangle =0$ for any $\bv\in \mathbf{RM}$.
Let $(\bu_{c_0},\eta_{c_0},\xi_{c_0})$ denote the unique weak solution to the problem \reff{2.21}-\reff{2.23}. Then there exists
$(\bu_*, \eta_*,\xi_*)\in \bL^\infty(0,T;\bH^1_\perp(\Ome))\times L^\infty(0,T;L^2(\Ome))\times
%L^\infty(0,T; L^2(\Ome)))\times L^\infty(0,T;L^2(\Ome))\cap L^2(0,T;H^1(\Ome)\times L^\infty(0,T;L^2(\Ome))$
L^2(0,T; L^2(\Ome))$
%\times L^2(0,T;H^1(\Ome)\times L^\infty(0,T;L^2(\Ome))$
such that $(\bu_{c_0},\eta_{c_0},\xi_{c_0})$ converges weakly to $(\bu_*, \eta_*,\xi_*)$ as $c_0\to 0$.
\end{theorem}

\begin{proof}
It follows immediately from \reff{2.59}-\reff{2.61} and Korn's inequality that
\begin{itemize}
\item $\bu_{c_0}$ is uniformly bounded (in $c_0$) in $\bL^\infty(0,T; \bH^1_\perp(\Ome))$;

\item $\sqrt{\kappa_2} \eta_{c_0}$ is uniformly bounded (in $c_0$) in
$L^\infty(0,T; L^2(\Ome))\cap H^1(0,T; H^{1}(\Ome)')$;

\item $\sqrt{\kappa_3}\xi_{c_0}$ is  uniformly bounded (in $c_0$) in $L^\infty(0,T; L^2(\Ome))$;

\item $\xi_{c_0}$ is uniformly bounded (in $c_0$) in $L^2(\Ome)$.

%\item $p_{c_0}$ is uniformly bounded (in $c_0$) in $L^\infty(0,T; L^2(\Ome)) \cap L^2(0,T; H^1(\Ome))$.
%\item $p_{c_0}$ is uniformly bounded (in $c_0$) in $L^2(0,T; H^1(\Ome))$;

%\item $q_{c_0}$ is uniformly bounded (in $c_0$) in $L^\infty(0,T; L^2(\Ome))$.
\end{itemize}
On noting that $\lim_{c_0\to 0}\kappa_1=\frac{1}{\alpha}$,
$\lim_{c_0\to 0}\kappa_2=\frac{\lam}{\alpha^2}$ and $\lim_{c_0\to 0}\kappa_3=0$, by the weak compactness of reflexive Banach spaces and Aubin-Lions Lemma(cf.  \cite{dautray}), we know that there exist
$(\bu_*, \eta_*,\xi_*)\in \bL^\infty(0,T;\bH^1_\perp(\Ome))
\times L^\infty(0,T;L^2(\Ome))\times L^2(0,T; L^2(\Ome))$ and
a subsequence of $(\bu_{c_0},\eta_{c_0},\xi_{c_0})$ (still denoted by the
same notation) such that as $c_0\to 0$ (a subsequence of $c_0$, to be exact)
\begin{itemize}
\item $\bu_{c_0}$ converges to $\bu_*$ weak $*$ in $\bL^\infty(0,T; \bH^1_\perp(\Ome))$ and
weakly in $\bL^2(0,T; \bH^1_\perp(\Ome))$;

\item $\sqrt{\kappa_2} \eta_{c_0}$ converges to $\frac{\sqrt{\lam}}{\alpha} \eta_*$ weak $*$ in
$L^\infty(0,T; L^2(\Ome))$ and strongly in $L^2(\Ome)$;

\item $\kappa_3\xi_{c_0}$ converges to $0$ strongly in  $L^2(\Ome)$;

\item $\xi_{c_0}$ converges to $\xi_*$ weakly in  $L^2(\Ome)$.
\end{itemize}

Using \reff{2.30}, we deduce that
 $\mathcal N(\nabla\bu_{c_0})$ converges to $\mathcal N(\nabla\bu_*)$ weak $*$ in $\bL^\infty(0,T; \bH^1_\perp(\Ome))$ and weakly in $\bL^2(0,T; \bH^1_\perp(\Ome))$.

Then, setting $c_0\to 0$ in \reff{2.35}-\reff{2.39} yields
\begin{alignat*}{2}
\bigl(\mathcal{N}(\nabla\bu_*), \vepsi(\bv) \bigr)-\bigl( \xi_*, \div \bv \bigr)
&= (\bbf, \bv)+\langle \bbf_1,\bv\rangle &&\qquad\forall \bv\in \bH^1(\Ome), \\
\bigl(\div\bu_*, \varphi \bigr) &= \frac{1}{\alpha}\bigl(\eta_*, \varphi \bigr)
&&\qquad\forall \varphi \in L^2(\Ome),  \\
\bigl((\eta_*)_t, \psi \bigr)_{\rm dual}
%+\frac{1}{\mu_f} \bigl(K(\nab (\alpha^{-1}\xi_* +\lambda\alpha^{-2}\eta_*) &-\rho_f\bg), \nab \psi \bigr) &&\\
+\frac{1}{\mu_f} \bigl(K(\nab p_* &-\rho_f\bg), \nab \psi \bigr) &&\\
&= (\phi, \psi)+\langle \phi_1,\psi\rangle &&\qquad\forall \psi \in H^1(\Ome) ,  \\
p_*:=\frac{1}{\alpha}\xi_* +\frac{\lam}{\alpha^2}\eta_*, \qquad
&q_*:=\frac{1}{\alpha} \eta_* &&\qquad\mbox{in } L^2(\Ome),  \\
&\bu_*(0) = \bu_0, \qquad  &&  \\
%q_*(0)=q_0:=\div \bu_0,\quad \quad
\eta_*(0)= \eta_0 :&=\alpha q_0.  &&
\end{alignat*}
That is,
\begin{alignat*}{2}
\bigl(\mathcal{N}(\nabla\bu_*), \vepsi(\bv) \bigr)-\bigl( \xi_*, \div \bv \bigr)
&= (\bbf, \bv)+\langle \bbf_1,\bv\rangle
&&\quad\forall \bv\in \bH^1(\Ome), \\
\bigl(\div\bu_*, \varphi \bigr) &= \bigl(q_*, \varphi \bigr)
&&\quad\forall \varphi \in L^2(\Ome),  \\
\alpha \bigl((q_*)_t, \psi \bigr)_{\rm dual}
+\frac{1}{\mu_f} \bigl(K(\nab (\lambda\alpha^{-1}q_* +\alpha^{-1}\xi_*) &-\rho_f\bg), \nab \psi \bigr) && \\
&= (\phi, \psi)+\langle \phi_1,\psi\rangle  &&\quad\forall \psi \in H^1(\Ome) ,  \\
p_* &:=\frac{1}{\alpha} \Bigl(\xi_* +\lam q_* \Bigr) &&\quad\mbox{in } L^2(\Ome), \\
% \quad\mbox{or}\quad &\xi_*=\alpha p_*-\lam q_*, \quad  &&  \\
%\bu_*(0) = \bu_0, \quad
q_*(0)=q_0 &:=\div \bu_0.\quad   &&
\end{alignat*}
Hence, $(\bu_*, \eta_*,\xi_*)$ is a weak solution of the nonlinear Biot's consolidation model (cf. \cite{yd04b,fh10}). Using Theorem \ref{thm2.5},  we conclude that the whole sequence
$(\bu_{c_0},\eta_{c_0},\xi_{c_0})$ converges to $(\bu_*, \eta_*,\xi_*)$ as $c_0\to 0$ in the above sense. The proof is complete.
\end{proof}
%%%%%%%
\section{Conclusion}
In this paper, we deal with the nonlinear poroelasticity model with the constitutive relation $\tilde{\sigma}(\mathbf{u})=\mu\tilde{\varepsilon}(\mathbf{u})+\lambda tr(\tilde{\varepsilon}(\mathbf{u}))\mathbf{I}$, where the deformed Green strain tensor is $ \tilde{\varepsilon}(\bu)=\dfrac{1}{2}(\nabla\bu+\nabla^{T}\boldsymbol{u}+2\nabla^{T}\bu\nabla\bu)$. To better describe the proccess of deformation and diffusion underlying in the original model, we firstly reformulate the nonlinear poroelasticity by a multiphysics approach£¬which transforms the nonlinear fluid-solid coupling problem to a fulid-fluid coupling problem. Then, we adopt the similar technique of proving the well-posedness of nonlinear Stokes equations to prove the existence and uniqueness of weak solution of a nonlinear poroelasticity model. And we strictly prove the growth, coercivity and monotonicity of the nonlinear stress-strain relation by using the Cauchy-Schwarz inequality and some other inequalities, give the energy estimates and use Schauder's fixed point theorem to show the existence and uniqueness of weak solution of the nonlinear poroelasticity model. Besides, we prove that the weak solution of nonlinear poroelasticity model converges to the nonlinear Biot's consolidation model as the constrained specific storage coefficient trends to zero. To the best of our knowledge, it is the first time to prove the the existence and uniqueness of weak solution based on a multiphysics approach without any adding codition on the nonlinear stress-strain relation. Besides, we find out that the multiphysics approach is key to propose a stable numerical method for the nonlinear poroelasticity, and we will give the main results of some relative numerical method for the nonlinear poroelasticity in the future work.


\begin{thebibliography}{99}
	
	
\bibitem{biot} M. Biot, {\em Theory of elasticity and consolidation for a porous anisotropic media}, Journal of Applied Physics, 1955, 26(2): 182-185.

%\bibitem{brenner} S. Brenner, {\em A nonconforming mixed multigrid method for the pure displacement problem in planar linear elasticity}, SIAM Journal on Numerical Analysis, 1993, 30: 116-135.

\bibitem{bs08} S. Brenner, L. Scott, {\em The Mathematical Theory of Finite Element Methods},  third edition, Springer, 2008.


\bibitem{cia} P. Ciarlet, {\em The Finite Element Method for Elliptic Problems}, North-Holland, Amsterdam, 1978.

\bibitem{coussy04} O. Coussy, {\em Poromechanics}, Wiley \& Sons, England, 2004.

\bibitem{dautray} R. Dautray, J. Lions, {\em Mathematical Analysis and Numerical Methods for Science and Technology. Vol. 1}, Springer Verlag, 1990.

\bibitem{de86}
M. Doi, S. Edwards, {\em The Theory of Polymer Dynamics},
Clarendon Press, Oxford, 1986.



\bibitem{20210820} L. Evans, {\em Partial Differential Equations}, American Mathematical Society, 2016.

\bibitem{fglarxiv} X. Feng, Z. Ge, Y. Li, {\em Multiphysics finite element methods for a poroelasticity model}, arXiv:1411.7464, [math.NA], (2014).

\bibitem{fgl14} X. Feng, Z. Ge, Y. Li, {\em Analysis of a multiphysics finite element method for a poroelasticity model}, IMA Journal of Numerical Analysis, 2018, 38: 330-359.

\bibitem{fh10} X. Feng, Y. He, {\em Fully discrete finite element
	approximations of a polymer gel model}, SIAM Journal Numerical Analysis, 2010, 48: 2186-2217.

\bibitem{7} M. Ferronato, N. Castelletto, G. Gambolati, {\em A fully coupled 3-D mixed finite element model of Biot consolidation}, Journal of Computational Physics, 2010, 229(12): 4813¨C4830.

%\bibitem{GT} D. Gilbarg, N. Trudinger, {\em Elliptic
%Partial Differential Equations of Second Order}, Second Edition,
%Springer, New York, 2000.


\bibitem{8} D. Gawin, P. Baggio, B. Schrefler, {\em Coupled heat, water and gas flow in deformable porous media}, International Journal for Numerical Methods in Fluids, 1995, 20: 969¨C978.

\bibitem{gra} V. Girault, P. Raviart, {\em
	Finite Element Method for Navier-Stokes Equations: theory and algorithms},
Springer-Verlag, Berlin, Heidelberg, New York, 1981.

	
\bibitem{hamley07} I. Hamley, {\em Introduction to Soft Matter}, John Wiley \& Sons, 2007.
	
\bibitem{6} J. Hudson, O. Stephansson, J. Andersson, C. Tsang, L. Ling, {\em Coupled T¨CH¨CM Issues related to radioactive waste repository design and performance}, International Journal of Rock Mechanics and Mining Sciences, 2001, 38: 143¨C161.
	
\bibitem{LSU} O. Lady\v zenskaja, V. Solonnikov, N. Uarlceva,
{\em  Linear and quasilinear equations of parabolic type}, Translations of Mathematical Monographs, Vol. 23, American Mathematical Society, 1967.

\bibitem{murad} M. Murad, A. Loula, {\em
	Improved accuracy in finite element analysis of Biot's consolidation problem}, Computer Methods in Applied Mechanics and Engineering, 1992, 95: 359-382.

%\bibitem{pw09} P. Phillips, M. Wheeler, {\em Overcoming the problem of locking
%in linear elasticity and poroelasticity: an heuristic approach}. Computational Geosciences, 2009, 13: 5--12.

		\bibitem{10} D. Nemec, J. Levec, {\em Flow through packed bed reactors: 1. single-phase flow}, Chemical Engineering Science, 2005, 60: 6947-6957.
	
	
	\bibitem{3} W. Pao, R. Lewis, I. Masters, {\em A fully coupled hydro-thermo-poro-mechanical model for black oil reservoir simulation}, International Journal for Numerical and Analytical Methods in Geomechanics, 2001, 25: 1229¨C1256.


\bibitem{pw07} P. Phillips, M. Wheeler, {\em A coupling of mixed and continuous Galerkin finite element methods for poroelasticity I: the continuous in time case}, Computational Geosciences, 2007, 11: 131-144.



\bibitem{20210819} R. Showalter, {\em Diffusion in poro-elastic media}, Journal of Mathematical Analysis and Applications, 2000, 251: 310-340.


\bibitem{temam}  R. Temam, {\em Navier-Stokes Equations},
Studies in Mathematics and its Applications, Vol. 2, North-Holland, 1977.

\bibitem{2} A. Vuong, L. Yoshihara, W. Wall, {\em A general approach for modeling interacting flow through porous media under finite deformations}, Computer Methods in Applied Mechanics and Engineering, 2015, 283: 1240¨C1259.

\bibitem{yd04b} T. Yamaue, M. Doi, {\em Swelling dynamics of constrained thin-plate under an external force}, Physical Review E, 2004, 70: 011401.

\bibitem{201912095}  Z. Zhu,  {\em Discussion of nonlinear strain}, Advances in Mechanics, 1983, 13(3): 259-272.

%\bibitem{20210710}  Z. Ge, H. Li, T. Li, H. Lou, {\em Multiphysic nonlinear finite element method for a nonlinear poroelasticity model with small strain}, Preprint and submitted, 2021.



\end{thebibliography}
\end{document}